\documentclass[11pt,twoside,reqno]{amsart}

\usepackage{amsmath,amsfonts,amsthm,amssymb}
\usepackage{graphicx,psfrag,overpic}
\usepackage{bm,enumerate}
\usepackage{cases}
\usepackage[all,cmtip,line]{xy}
\usepackage{array,CJK}
\usepackage{array,setspace}
\usepackage{color}
\usepackage{yhmath}

\numberwithin{equation}{section}
\numberwithin{figure}{section}

\usepackage{graphics}
\usepackage{epsfig}
\newtheorem{claim}{\bf \t}[part]

\newtheorem{theorem}{Theorem}[section]

\newtheorem{lemma}{Lemma}[section]

\newtheorem{remark}{Remark}[section]
\newtheorem{definition}{Definition}[section]
\newtheorem{thm}{Theorem}[section]
\newtheorem{lem}[thm]{Lemma}

\newcommand{\qnt}[1]{\left(#1\right)}
\DeclareMathOperator*{\dist}{dist}
\newcommand{\dif}{\mathrm{d}}

\title{An Inverse Problem for Multi-Dimensional Piston Models with Large Velocity Variations}

\author{Dian Hu$^{1}$, Qianfeng Li$^{2,3}$, Yongqian Zhang$^{3}$}

\thanks{$^{1}$ School of Sciences, East China University of Science and Technology, Shanghai, 200237, P.R. China. Email: \texttt{hudianaug@qq.com}}

\thanks{$^{2}$ Department of Mathematics, Friedrich-Alexander-Universität Erlangen-Nürnberg, Cauerstr. 11, 91058 Erlangen, Germany. Email: \texttt{qianfeng.li@fau.de}}

\thanks{$^{3}$ School of Mathematical Sciences, Fudan University, Shanghai 200433, P.R. China. Email: \texttt{yongqianz@fudan.edu.cn}}

\thanks{Corresponding author: Qianfeng Li, \texttt{qianfeng.li@fau.de}}


\begin{document}
	\begin{abstract}
		\,\,When a circular symmetric piston suddenly expands into a still gas, a leading shock wave is generated. This paper investigates an inverse problem of reconstructing the trajectory of the piston  from the given leading shock front and the given initial flow conditions. We observe that in piston models, as the initial density goes to zero, the piston approaches the shock front; however, in the region between the piston and the shock front, the strict hyperbolicity of the system degenerates. By applying asymptotic analysis, we provide quantitative characterizations of the distance between the piston and the shock front, and the degeneration of strict hyperbolicity. Consequently, by designing appropriate a priori assumptions to balance the benefits and drawbacks arising as the initial density approaches zero, we employ the method of characteristics to 
    prove the global-in-time existence of the piecewise smooth solution for this inverse problem. In particular, the resulting flow structure exhibits significant velocity variations. 
	\end{abstract}
	
	\maketitle
	\section{Introduction}
	
	Piston model is not only a basic prototype model in gas dynamic \cite{Courant1948}, but also an efficient approximation for hypersonic flow past slender bodies \cite{KuangJie2021,Tsien1946}. When a sphere expanding into a still gas with constant expanding speed, Taylor \cite{Taylor1946}, by numerical integration, firstly gives the self-similar flow configuration with a spherical leading shock, and Chen \cite{Chen2003JDE} carries out the analytical proof for such self-similar flow configuration. Besides there are many works devoted to the direct problem of determining flow fields and leading shock front with given initial flow fields and the given piston trajectory, for instance, see \cite{Wangzejun2004,Wang2004ACTA} for the local piecewise smooth solution around the singular point $r=0,$ see \cite{Wang2005DCDS,Wangzejun2004,Wangzejun2008Global} for the global admissible BV and $L^{\infty}$ weak solution, and see \cite{Ding2013ZAMP,Dingmin2013An,Lai2023EJAM,Lai2020EJAM} for the local and global piecewise smooth solution in relativistic cases. 
		
	In the paper, we are concerned with an inverse problem for a multi-dimensional piston moving into still gas, where we want to design piston trajectory such that the spherical shock front produced matches the given shock trajectory. The flow is governed by 
	\begin{equation}\label{E1}
	\left\{
	\begin{aligned}
	&\pi_t+\nabla\cdot(\pi \vec{V})=0,\\
	&(\pi\vec{V})_t+\nabla\cdot(\pi\vec{V}\otimes\vec{V}+pI)=0
	\end{aligned}
	\right.
	\end{equation} 
	where $\vec{V}\in \mathbb{R}^3$ and $\pi\in \mathbb{R}^+$ denote the velocity and the density respectively, and $\displaystyle p=A\pi^{\gamma}$ denotes the pressure with constants $A>0,\gamma\in(1,3).$ For simplification, we set $A=1$ in the sequel.  
	
	The initial data is given by 
	\begin{equation}\label{E2}
	(\pi,\vec{V})(\vec{x},0)=(\rho_{\infty},\vec{0}),~\vec{x}\in \mathbb{R}^3,
	\end{equation}
	with $\rho_{\infty}>0$ being constant. That is gas is static at the beginning.
	
 	The spherical shock trajectory is given by $$\mathsf{S}:=\{(\vec{x},t):\big|\vec{x}\big|=s(t), t>0\}$$ with $s(t)\in C^2(\mathbb{R}^{+}), s(0)=0.$ 
	
	As both initial data and the given shock trajectory are with spherical symmetric, we would find spherical solution for the dynamical process. Thus we assume that the piston has a trajectory $$\mathsf{P}:=\{(\vec{x},t):\big|\vec{x}\big|=b(t), t>0\}.$$
	
	Set $$\Omega_+:=\{(\vec{x},t):b(t)<\big|\vec{x}\big|<s(t),t>0\}$$ and $$\Omega_-:=\{(\vec{x},t):\big|\vec{x}\big|>s(t),t>0\}.$$ See Figure \ref{fig:sperical-piston-problem}.

	\begin{figure}
		\centering
		\includegraphics[width=0.5\linewidth]{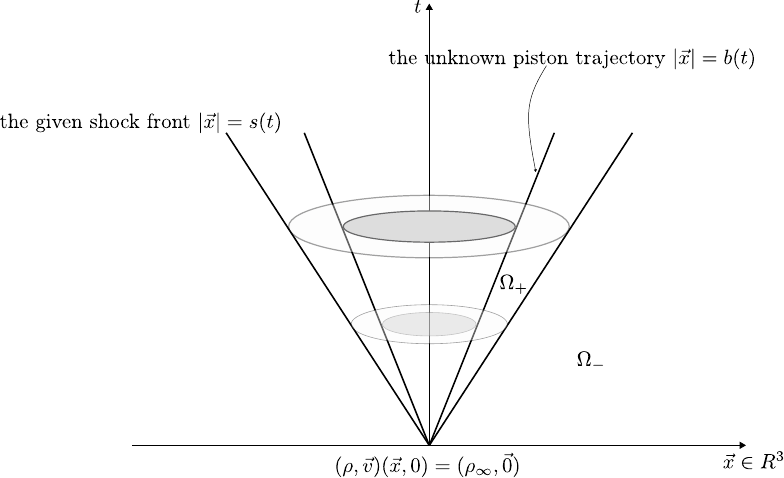}
		\caption{The schematic diagram for spherical piston inverse problem}
		\label{fig:sperical-piston-problem}
	\end{figure}
	
	Then we consider \eqref{E1} and \eqref{E2} in $\Omega_+\cup\Omega_-$ with the following conditions:
	\begin{equation}\label{E3}
	b^{\prime}(t)=\vec{V}\cdot\frac{\vec{x}}{\big|\vec{x}\big|}, \text{ on } P,
	\end{equation}
	and 
	\begin{equation}\label{E4}
	\left\{
	\begin{aligned}
	&[\pi]s^{\prime}(t)=[\pi \vec{V}]\cdot\frac{\vec{x}}{\big|\vec{x}\big|},\\
	&[\pi \vec{V}]s^{\prime}(t)=[\pi\vec{V}\otimes\vec{V}+pI]\cdot\frac{\vec{x}}{\big|\vec{x}\big|},\\
	&\pi_{+}>\pi_{-}, \text{ on } S.
	\end{aligned}
	\right.
	\end{equation}
	Here and in the sequel, we use the following notations for any $(\vec{x}_0,t_0)\in S,$
	\begin{equation*}
	\begin{aligned}
	&[A](\vec{x}_0,t_0):=A_{+}-A_{-},\\
	&A_{+}=\lim_{(\vec{x},t)\to(\vec{x}_0,t_0),(\vec{x},t)\in\Omega_+}	A(\vec{x},t),\\
	&A_{-}=\lim_{(\vec{x},t)\to(\vec{x}_0,t_0),(\vec{x},t)\in\Omega_-}	A(\vec{x},t),
	\end{aligned}
	\end{equation*}
	and $A$ can be one of the quantities $\pi, \pi\vec{V}, \pi\vec{V}\otimes\vec{V}+pI$ in \eqref{E4},  \eqref{E3} is the non-slip condition and the first two equations in \eqref{E4} are called the Rankine-Hugoniot condition and the last one in \eqref{E4} is the entropy condition.
	
	Our problem is to determine $b(t)\in C^2(\mathbb{R}^+)$ and $(\pi,\vec{V})\in C^1(\Omega_+\cup\Omega_-)$ for prescribed $\rho_{\infty}$ and prescribed $s(t)\in C^2(\mathbb{R}^+)$. Due to $\displaystyle(\pi,\vec{V})\big|_{\Omega_-}=(\rho_\infty,\vec{0})$ and the  spherical symmetry setting, we rewrite the problem \eqref{E1}-\eqref{E4} as follows. More precisely, let 
    \begin{equation}\label{coodinatestransform1}
    r=\big|\vec{x}\big|, \pi(\vec{x},t)={\rho}(r,t), \vec{V}(\vec{x},t)={v}(r,t)\frac{\vec{x}}{\big|\vec{x}\big|},    
    \end{equation}
    and denote as
    $$\Omega=\mathcal{M}(\Omega_+), \mathsf{P}=\mathcal{M}(P), \mathsf{S}=\mathcal{M}(S),$$
    where $\mathcal{M}$ is the map $\displaystyle (\vec{x},t)\mapsto (\big|\vec{x}\big|,t),$  
    then problem \eqref{E1}-\eqref{E4} is equivalent to  
	\begin{equation}\label{EE1}
	\left\{
	\begin{aligned}
	&\rho_t+(\rho v)_r+\frac{2\rho v}{r}=0,\\
	&(\rho v)_t+(\rho v^2+\rho^\gamma)_r+\frac{2\rho v^2}{r}=0, \text{ in } \Omega,
	\end{aligned}
	\right.
	\end{equation}
	with 
	\begin{equation}\label{EE3}
	b^{\prime}(t)=v(b(t),t), \text{ on } \mathsf{P},
	\end{equation}
	and 
       \begin{equation}\label{EE4}
	\left\{
	\begin{aligned}
	&(\rho-\rho_{\infty})s^{\prime}(t)=\rho v,\\
	&\rho vs^{\prime}(t)=(\rho v^2+\rho^\gamma-\rho_\infty^\gamma),\\
	&\rho>\rho_{\infty}, \text{ on } \mathsf{S}.
	\end{aligned}
	\right.
	\end{equation} 
	Here and in the sequel, $(\rho,v)\big|_\mathsf{S}$ means the trace of $(\rho,v)\big|_{\Omega}$ on the shock $\mathsf{S}.$  Therefore, we focus on problem \eqref{EE1}-\eqref{EE4} in the sequel. 
	
	For the one-dimensional inverse piston problem, when the prescribed initial data is close to a constant state and the prescribed shock trajectory is near a straight line, Li and Wang \cite{Wanglibin2007}, Wang \cite{WangLibin2014}, Wang and Wang  \cite{Wanglibin2019}  apply characteristic method to establish the global smooth piston trajectory and the global piecewise smooth flow field. 
    
    Some other well-developed inverse problems in hyperbolic conservation law include the initial data identification in Burgers equation \cite{allahverdi2016numerical,castro2008alternating,castro2010optimal,colombo2020initial,esteve2020inverse,gosse2017filtered,liard2021initial,liard2023analysis},
    and reconstruction of the shape of the obstacles in the context of supersonic flow past obstacles with an attached leading supersonic shock. The latter can be done either from the prescribed location of the leading shock and the prescribed incoming flow \cite{hu2024inverse,li2022inverse,li2006global,wang2011direct} or from the prescribed pressure on the obstacle's surface and the prescribed incoming flow \cite{goldsworthy1952supersonic,pu2023inverse}. Considering the initial data identification in Burgers equation, the sufficient and necessary condition for attainable final states is established in \cite{colombo2020initial,gosse2017filtered}, and the collection of the initial states corresponding to an attainable final state is fully characterized in \cite{esteve2020inverse,liard2021initial}. Moreover, when formulated as non-smooth optimization problems, initial data identification has been widely studied using numerical methods, as shown in \cite{allahverdi2016numerical,castro2008alternating, castro2010optimal, liard2023analysis} and the references therein.
    

    In the article, to establish the global solvability of problem \eqref{EE1}-\eqref{EE4}, we analyze the Riemann invariants and their derivatives along the forward and backward characteristics emitting from the shock front $\mathsf{S}$ and getting into $\Omega.$ There are some important observations in our setting:

     \textbf{(1)} The given shock's speed could vary in a wide range, which may result in significant variations in the flow fields in $\Omega$. Consequently, the behavior of characteristics is complicated in $\Omega,$ and the perturbation analysis presented in \cite{Wanglibin2007,WangLibin2014,Wanglibin2019} fails here. Fortunately, according to some finer analysis, we figure out that before arriving at the piston, the characteristics emitting at the point $(s(t_0),t_0)\in \mathsf{S} $ locate in such a narrow temporal strip $\displaystyle\Omega\cap \{(1-l(\rho_\infty))t_0<t<(1+l(\rho_\infty))t_0\}$, in which the flow states are almost constant, and the trajectory of the shock and the piston and the characteristics are almost straight. Here $\displaystyle l(\rho_\infty)>0, \lim_{\rho_\infty\to 0+}l(\rho_\infty)=0$ and the limit process is uniform with respect to $(r,t)\in \mathbb{R}^+\times \mathbb{R}^+.$   

     \textbf{(2)} Since the characteristics are confined to such a narrow temporal strip, when integrating \eqref{E10} and \eqref{E11} to do $C^1$ estimates, the effect of the geometric source term  $\displaystyle 2\rho v/r, 2\rho v^2/r$ in \eqref{EE1} is bounded by $$\displaystyle \int_{(1-l(\rho_{\infty}))r_0}^{(1+l(\rho_{\infty}))r_0}\frac{1}{r} ~\dif r,$$ which is sufficiently close to $0$ when $\rho_\infty$ sufficiently close to $0.$ Moreover, these source terms are singular at $t=0.$ To avoid the difficulty, we assume the shock expands with constant speed near $t=0$, and consider the self-similar solution. By some delicate analysis, we figure out that the self-similar solution satisfies the required a priori assumption \eqref{EE47}. 

     \textbf{(3)} There is a loss of strict hyperbolicity of \eqref{EE1} in $\Omega$ as $\rho_{\infty}$ tends to $0.$ Thus, we need to balance the required smallness of $\rho_\infty$ and the loss of strict hyperbolicity of the system. To achieve this,  we apply a suitable asymptotic expansion of the solution near $\rho_{\infty}=0$ to conduct some fine estimates. Moreover, we remark that the a priori assumption need to be designed carefully.

    Now, we state the main results as follows.
    
	\begin{theorem}\label{thm1}
		For any given positive constants $\kappa_1,\kappa_2,\kappa_3,\kappa_{4}, \varpi_0$ there exists $\epsilon>0$ such that if the given initial data $(\rho_{\infty},0)$ and the given shock trajectory $\displaystyle \mathsf{S} =\{(r,t):r=s(t), t>0\}$ jointly satisfy $\displaystyle 0<\rho_{\infty}<\epsilon$ and 
        \begin{equation*}
               \text{(A1) }\kappa_1<s^{\prime}(t)<\kappa_2;\quad \text{(A2) }  s''(t)=0, t\in(0,\kappa_{3});\quad \text{(A3) } \sup_{t\in \mathbb{R}^{+}} \big|ts''(t)\big|<\kappa_4\rho^{\varpi_0}_{\infty};
        \end{equation*}
		then problem \eqref{EE1}-\eqref{EE4} globally admits $b(t)\in C^2(\mathbb{R}^{+})$ and $(\rho,v)\in C^1(\Omega).$
	\end{theorem}
	
The remaining part is organized as follow. In Section 2, we consider the given shock moving with constant speed where the flow field is with self-similar structure. By analyzing the Rankine-Hugoniot condition and the monotonic properties of such self-similar flow, we establish the asymptotic expansion with respect to $\displaystyle\rho_\infty$ for the flow field and its derivative. In Section 3, we derive asymptotic expansion results for the solution near the shock in the case with variable shock speed. Additionally, we analyze how the flow states $\rho,v$,  as well as $\lambda_{\pm}-s'(t),$ the differences between the characteristic values and the shock speed, depend on the variation of the Riemann invariants $w_{\pm}.$ Furthermore, we provide estimates for the derivatives of the solution on the given shock. 
In Section 4, we first introduce the Riemann invariants $w_\pm$ to rewrite system \eqref{EE1} into a diagonal form. Next, we demonstrate that any characteristic curve, before reaching the shock or the piston, remains confined to a narrow temporal strip. Utilizing these narrow estimates, we then analyze the Riemann invariants along the characteristic curves to establish the $C^1$ a priori estimates. Finally, we use these results to prove Theorem \ref{thm1}.

\section{Shock wave moving with constant speed}
When the shock wave expands into still gas with constant speed, the flow field is with self-similar configuration, and its existence has been established in \cite{Chen2003JDE,Wangzejun2008Global}. In the section, we will give fine estimates on the variation of such self-similar configuration. 

Under the self-similar configuration assumption, setting
\begin{equation}\label{coodinatestransform2}
    \sigma=\frac{r}{t}, ~\rho(r,t)=\varrho(\sigma), ~u(r,t)=\vartheta(\sigma), 
\end{equation}
and denoting as
\begin{equation}
    \varrho_\sigma=\frac{\dif \varrho}{\dif \sigma}, ~\vartheta_\sigma=\frac{\dif \vartheta}{\dif \sigma},
\end{equation}
we rewrite problem \eqref{EE1}-\eqref{EE4} into 
\begin{equation}\label{eq:Constantmovingspeedcase}
    \left\{
    \begin{aligned}
        &\sigma(\vartheta-\sigma)\varrho_{\sigma}+\sigma\varrho\vartheta_\sigma+2\varrho\vartheta=0,\\
        &(\vartheta-\sigma)\vartheta_\sigma+\gamma\varrho^{\gamma-2}\varrho_{\sigma}=0, \text{ for }\sigma\in(b_0,s_0),\\
        & \vartheta(\sigma)=b_0, \text{ for } \sigma=b_0,\\
        &(\varrho-\rho_{\infty})s_0=\varrho\vartheta,\\
        &\varrho\vartheta s_0 =\varrho\vartheta^2+\varrho^{\gamma}-\varrho^{\gamma}_{\infty},\\
        &\varrho>\rho_{\infty}, \text{ for } \sigma=s_0.
    \end{aligned}
    \right.
\end{equation}
Here, $s_0\in(\kappa_1,\kappa_2)$ is the given shock speed, $b_0, (\varrho,\vartheta)\big|_{(b_0.s_0)}$ denote the corresponding piston speed and the flow field to be determined. 

We first verify Rankine-Hugoniot condition and entropy condition admit unique $(\varrho(s_0),\vartheta(s_0))$ for $\rho_\infty$ close to zero, and then figure out the following two inequalities on  $\varrho(s_0),\vartheta(s_0)$.

\begin{lem}\label{lem:RHC}
For $\rho_\infty$ close enough to zero, for any given $s_0\in(\kappa_1,\kappa_2),$ Rankine-Hugoniot condition and entropy condition, i.e. the last three formulas in \eqref{eq:Constantmovingspeedcase}, admit unique $(\varrho(s_0),\vartheta(s_0)).$ Furthermore, there hold that
    \begin{equation}\label{eq:RHC}
    (\vartheta-\sigma)^2-\gamma\varrho^{\gamma-1}\big|_{\sigma=s_0}<0,
    \end{equation}
    and 
    \begin{equation}\label{eq:susgn0}
        \vartheta-\sigma\big|_{\sigma=s_0}<0.
    \end{equation}
    Here $\kappa_1,\kappa_2$ are given in Theorem \ref{thm1}.
\end{lem}
\begin{proof}
    It directly follows from Rankine-Hugoniot condition and entropy condition that for $\sigma=s_0,$
    \begin{equation}\label{eq:E1}
     	\left\{
     	\begin{aligned}
     	&\vartheta=\sqrt{\frac{(\varrho-\rho_{\infty})(\varrho^{\gamma}-\rho^{\gamma}_{\infty})}{\rho_{\infty}\varrho}},\\
     	&\vartheta=\frac{(\varrho-\rho_{\infty})}{\varrho}s_0,\\
        &\varrho>\rho_\infty.
     	\end{aligned}
     	\right.
     	\end{equation}
      Then eliminating $\vartheta$ in above leads to for $\sigma=s_0,$
      \begin{equation}\label{eq:E2}
          s_0^2=\frac{\varrho(\varrho^{\gamma}-\rho^{\gamma}_{\infty})}{(\varrho-\rho_{\infty})\rho_{\infty}}, ~~\rho>\rho_\infty
          .
      \end{equation}
      Moreover, by setting $\displaystyle k=\frac{\varrho(s_0)}{\rho_{\infty}},$ we rewrite \eqref{eq:E2} as
      \begin{equation}\label{eq:RTT2}
          s_0^2\rho_\infty^{1-\gamma}=\frac{k(k^\gamma-1)}{k-1}:=f(k),~~ k>1.
      \end{equation}
     
     Note the facts that $ f'(k)>0 \text{ for }k>1,$ and 
      \begin{equation}
     \lim_{k\to1}f(k)=\gamma, ~\lim_{k\to+\infty}f(k)=+\infty.
      \end{equation}
    Thus, we conclude from \eqref{eq:RTT2} that when $\displaystyle\rho_\infty\in(0,({\kappa_1}/{\gamma})^{\frac{1}{\gamma
       -1}}),$ the Rankine-Hugoniot condition and entropy condition admit unique $\varrho(s_0).$ Then inserting the obtained $\varrho(s_0)$ into the second formula in \eqref{eq:E1}, we obtain the unique $\vartheta(s_0).$

      We next prove the inequality \eqref{eq:RHC}. Combining \eqref{eq:E1} and \eqref{eq:E2},  direct computations show that for $\sigma=s_0,$
      \begin{equation}\label{eq:E3}
          \begin{aligned}
              (\vartheta-s_0)^2-\gamma\varrho^{\gamma-1}&=\frac{\rho^2_{\infty}}{\varrho^2}s_0^2-\gamma\varrho^{\gamma-1}= \frac{\rho^{\gamma-1}_{\infty}}{k(k-1)}f_1(k),
          \end{aligned}
      \end{equation}
   where $$f_1(k)=-\gamma k^{\gamma+1}+(\gamma+1)k^{\gamma}-1.$$ 
   
   For the function $f_1(k),$ direct computations show that $f_1(1)=0$ and for $k>1,$ 
   \begin{equation}
       \begin{aligned}
            f'_1(k)=(1-k)\gamma(\gamma+1)k^{\gamma-1}<0,
       \end{aligned}
   \end{equation}
which implies that  $\displaystyle f_1(k)<0 \text{ for } k>1.$ Therefore, inserting the entropy condition, i.e., $k>1$ and the fact that $\displaystyle f_1(k)<0 \text{ for } k>1$ into \eqref{eq:E3}, we obtain the second formula in \eqref{eq:RHC}. 

Finally, as for the inequality \eqref{eq:susgn0}, it follows form the second equality in \eqref{eq:E1} directly. The proof is complete. 
\end{proof}

The existence of the self-similar solution of problem \eqref{eq:Constantmovingspeedcase} has been established in \cite[Lemma 3.5]{Wang2005DCDS}. We study the monotonic relations of such self-similar flow field as follows. 

\begin{lem}\label{lem:Monotonic}
    For any given $s_0\in(\kappa_1,\kappa_2)$ with $\kappa_1,\kappa_2$ given in Theorem \ref{thm1}, let $(\varrho(\sigma),\vartheta(\sigma)), \sigma\in(s_0,b_0)$ be the $C^1$ solution of problem \eqref{eq:Constantmovingspeedcase}. Then there hold that for $\sigma\in(b_0,s_0),$
    \begin{equation}\label{eq:MonotonicRelationsforconstantcase}
        \varrho_\sigma<0, ~\vartheta_\sigma<0.
    \end{equation}
\end{lem}
\begin{proof}
    Define $$\mathsf{A}:=\{\sigma: \sigma\in [b_0,s_0),\vartheta(\sigma)=\sigma\},$$ and let
    $$\sigma_1:=\sup\mathsf{A}.$$
    Since $b_0\in \mathsf{A},$ $\sigma_1$ is well-defined. Moreover, by the continuity of $\vartheta(\sigma),\sigma\in(b_0,s_0)$, \eqref{eq:E1} implies that $\sigma_1<s_0.$ 
    
    The following proof is divided into two steps: first we show that \eqref{eq:MonotonicRelationsforconstantcase} holds on the interval $(\sigma_1,s_0);$ second we establish $\sigma_1=b_0.$

    \textbf{Step 1:}
    We claim that \eqref{eq:MonotonicRelationsforconstantcase} holds on the interval $(\sigma_1,s_0)$. Indeed, according to the first two equations of \eqref{eq:Constantmovingspeedcase}, a direct computation shows that 
    \begin{equation}\label{eq:EE3}
       \left\{
       \begin{aligned}
           &\sigma((\vartheta-\sigma)^2-\gamma\varrho^{\gamma-1})\varrho_{\sigma}=2\varrho\vartheta(\sigma-\vartheta),\\
           &\sigma((\vartheta-\sigma)^2-\gamma\varrho^{\gamma-1})\vartheta_{\sigma}=2\gamma\varrho^{\gamma-1}\vartheta.
       \end{aligned}
       \right.
    \end{equation}
    which together with  Lemma \ref{lem:RHC}, implies 
    \begin{equation}\label{eq:monotonicrelationats0}
       \left\{
       \begin{aligned}
           &\varrho_{\sigma}(s_0)<0,\\
           &\vartheta_{\sigma}(s_0)<0.
       \end{aligned}
       \right.
    \end{equation}
    Moreover, It directly follows from \eqref{eq:EE3} that 
    \begin{equation}
        \sigma^2((\vartheta-\sigma)^2-\gamma\varrho^{\gamma-1})^2\varrho_{\sigma}\vartheta_\sigma=2\gamma\varrho^\gamma\vartheta^2(\sigma-\vartheta),
    \end{equation}
    which together with $\displaystyle \sigma-\vartheta(\sigma)>0 \text{ for }\sigma\in (\sigma_1,s_0)$ following from the definition of $\sigma_1$, yields that
    \begin{equation}\label{eq:monotonicrelationontheinterval}
        \varrho_{\sigma}\vartheta_\sigma>0, \text{ for }\sigma\in (\sigma_1,s_0). 
    \end{equation}
    Therefore, combining \eqref{eq:monotonicrelationats0} and \eqref{eq:monotonicrelationontheinterval}, according to the continuity of $\varrho_\sigma$ and $\vartheta_\sigma,$ we arrive that \eqref{eq:MonotonicRelationsforconstantcase} holds on the interval $(\sigma_1,s_0)$.  

    \textbf{Step 2:} We claim that $\sigma_1=b_0.$ Otherwise, due to the continuity of $\varrho,$ there exists $\sigma_2\in(b_0,\sigma_1)$ such that 
    \begin{equation}\label{eq:localpositiveondensity} 
    \begin{aligned}    
    &\varrho>\varrho(s_0)/2>0, \text{ for }\sigma\in[\sigma_2,\sigma_1].
    \end{aligned}
    \end{equation}
    Here we apply the fact that \eqref{eq:MonotonicRelationsforconstantcase} holding on the interval $(\sigma_1,s_0)$ implies $\varrho(\sigma_1)>\varrho(s_0).$ 
    
    We next consider the self-similar flow field in the coordinates $(\vec{x},t)\in \mathbb{R}^3\times\mathbb{R}_+,$ and 
    derive the contradiction from \eqref{eq:localpositiveondensity}.  
    
    For fixed positive constant $\mathsf{T}$, define 
     $$\mathsf{A}_\mathsf{T}:=\{(\vec{x},t):t\in(0,\mathsf{T}), \big|\vec{x}\big|\in(\sigma_1t,s_0t)\cup(s_0t,s_0\mathsf{T})\},$$ 
     and
     $$\mathsf{B}_\mathsf{T}:=\{(\vec{x},t): t\in(0,\mathsf{T}), \big|\vec{x}\big|\in(b_0t,s_0t)\cup(s_0t,s_0\mathsf{T})\}.$$ 
    Recalling \eqref{coodinatestransform1} and \eqref{coodinatestransform2}, there holds that 
    \begin{equation}
        \pi(\vec{x},t)=\varrho(\big|\vec{x}\big|/t),~ \vec{V}(\vec{x},t)=\vartheta(\big|\vec{x}\big|/t)\frac{\vec{x}}{\big|\vec{x}\big|}
    \end{equation}
satisfy \eqref{E1}-\eqref{E4} with $s(t)=s_0t,~ b(t)=b_0t.$

     Then, integrating the first formula in \eqref{E1} over $\mathsf{A}_\mathsf{T}$ and $\mathsf{B}_\mathsf{T}$ respectively, and applying Stokes formulas, leads to 
     \begin{equation}\label{eq:Integrationbyparts}
        \begin{aligned}
            &\iiint_{\big|\vec{x}\big|\in(\sigma_1\mathsf{T}, s_0\mathsf{T})}\pi(\vec{x},\mathsf{T}) \dif \vec{x}= \iiint_{\big|\vec{x}\big|\in(0, s_0\mathsf{T})}\rho_\infty \dif \vec{x},\\
            &\iiint_{\big|\vec{x}\big|\in(b_0\mathsf{T}, s_0\mathsf{T})}\pi(\vec{x},\mathsf{T}) \dif \vec{x}= \iiint_{\big|\vec{x}\big|\in(0, s_0\mathsf{T})}\rho_\infty \dif \vec{x},
        \end{aligned}
     \end{equation}
     where we use Rankine-Hugoniot condition to eliminate the integration on the shock surface $\{(\vec{x},t): t\in(0,\mathsf{T}),\big|\vec{x}\big|=s_0t\}$ and use the slip boundary condition to eliminate the integration on the surfaces $\{(\vec{x},t):  t\in(0,\mathsf{T}),\big|\vec{x}\big|=\sigma_1t\}$ and $\{(\vec{x},t):  t\in(0,\mathsf{T}),\big|\vec{x}\big|=b_0t\}.$ 

     Note that $\displaystyle \varrho(\sigma)\geq 0, \sigma\in(b_0,s_0)$ and that $\displaystyle (b_0\mathsf{T},s_0\mathsf{T})=(b_0\mathsf{T},\sigma_2\mathsf{T})\cup(\sigma_2\mathsf{T},\sigma_1\mathsf{T})\cup(\sigma_1\mathsf{T},s_0\mathsf{T}).$ Substituting \eqref{eq:localpositiveondensity} into \eqref{eq:Integrationbyparts} leads to the contradiction. Thus, we conclude that $\sigma_1=b_0.$

    Finally, combining the two claims established in the preceding steps, we complete the proof.
\end{proof}

\begin{remark}\label{rem:localexsitence}
    As an additional observation, with a minor adjustment, regarding the claim in Step 1 in the proof of Lemma \ref{lem:Monotonic} as a priori estimates for problem \eqref{eq:Constantmovingspeedcase}, we can use continuity argument to establish the solvability of problem \eqref{eq:Constantmovingspeedcase} for $\rho_\infty$ suitably close to zero, referring to  \cite[Theorem 3.1]{li2024hypersonic} for details on the case that uniform hypersonic flows past a straight cone. Moreover, the argument in Step 2 in the proof of Lemma \ref{lem:Monotonic} guarantees the uniqueness of such a self-similar flow configuration.  
\end{remark}

Before proceeding further, we define $\mathcal{O}_+(\rho_\infty^\alpha), \alpha\in \mathbb{R},$ which is the most important notation in this article. This definition not only simplifies the presentation but also enhances the reader's understanding of the content.

\begin{definition}\label{def:defofO}
    Let $\mathcal{T}$ denote a positive quantity associated with $(r,t)$ and $\rho_\infty.$ We say  
    \begin{equation*}
        \mathcal{T}=\mathcal{O}_{+}(\rho^{\alpha}_{\infty}), 
    \end{equation*}
    for some $\alpha\in\mathbb{R},$ if there exist  positive constants $\mathsf{m}\in(0,1), \mathsf{M}\in(1,+\infty)$ independent of $(r,t)$ and $\rho_{\infty}$ such that when $\rho_\infty\in (0,\mathsf{m}),$
      	\begin{equation}\label{061711}
      	\mathsf{M}^{-1}\rho_\infty^\alpha<\mathcal{T}<\mathsf{M}\rho^{\alpha}_{\infty}.
      	\end{equation}
        Without confusion, we denote $\mathcal{O}_+(\rho_\infty^0)$ as $\mathcal{O}_+(1)$.
\end{definition}

Now, we are ready to figure out he quantitative properties of $(\varrho(\sigma), \vartheta(\sigma)), \sigma\in(b_0,s_0).$ 

\begin{thm}[Zeroth-order estimates for the self-similar flow field]\label{thm:VarianceofSelfsimilarSolution}
    For any given $s_0\in(\kappa_1,\kappa_2)$ with $\kappa_1,\kappa_2$ given in Theorem \ref{thm1}, let $(\varrho(\sigma),\vartheta(\sigma)), \sigma\in(s_0,b_0)$ be the $C^1$ solution of problem \eqref{eq:Constantmovingspeedcase}.
    For  $\rho_{\infty}$ close enough to zero, there holds that for $\sigma\in(b_0,s_0),$
    \begin{equation}\label{eq:Asympototicrelation}
     \begin{aligned}
        \varrho(\sigma)= \mathcal{O}_+(\rho_{\infty}^{\frac{1}{\gamma}}),  ~0<s_0-\vartheta(\sigma)\leq \mathcal{T}_1,
     \end{aligned}
    \end{equation}
   for some $\mathcal{T}_1 = \mathcal{O}_+(\rho_{\infty}^{\frac{\gamma-1}{\gamma}}).$
\end{thm}
\begin{proof}
    Using the notation $\displaystyle k=\frac{\varrho(s_0)}{\rho_{\infty}},$ we rewrite \eqref{eq:E2} into 
    \begin{equation}\label{eq:CRHC1}
        s_0^2=\frac{k(k^\gamma-1)}{k-1}\rho_\infty^{\gamma-1}.
    \end{equation}
    Since $s_0\in(\kappa_1,\kappa_2)$, and $k>1$ implied by  the entropy condition, we have that
    \begin{equation}
    \lim\sup_{\rho_{\infty}\to0+}k=\lim\inf_{\rho_{\infty}\to0+}k=+\infty,
    \end{equation}
    which implies 
    \begin{equation}\label{eq:CRHC2}
        \lim_{\rho_\infty\to0+}k=+\infty, ~\lim_{\rho_\infty\to0+}\frac{1}{k}=0. 
    \end{equation}

    Moreover, we derive from \eqref{eq:CRHC1} that 
    \begin{equation}
        k\rho_\infty^{\frac{\gamma-1}{\gamma}}=(\frac{k-1}{k}s_0^2+\rho_{\infty}^{\gamma-1})^{\frac{1}{\gamma}},
    \end{equation}
which together with \eqref{eq:CRHC2} and $\displaystyle s_0\in(\kappa_1,\kappa_2),$ gives that
\begin{equation}\label{eq:E7}
    k=\mathcal{O}_+(\rho^{\frac{1-\gamma}{\gamma}}_{\infty}),~~\text{i.e.,}~~ k^{-1}= \mathcal{O}_+(\rho^{\frac{\gamma-1}{\gamma}}_{\infty}). 
\end{equation}
That is, 
\begin{equation}\label{eq:E8}
    \varrho(s_0)=\mathcal{O}_+(\rho_{\infty}^{\frac{1}{\gamma}}).
\end{equation}

Recalling the the monotonic relation of $\vartheta$ in Lemma \ref{lem:Monotonic} and applying the second formula in \eqref{eq:E1}, a direct computation shows that
\begin{equation}\label{eq:E6}
0<s_0-\vartheta(\sigma)\leq s_0-\vartheta(s_0)=k^{-1}s_0,
\end{equation}
which proves the second formula in \eqref{eq:Asympototicrelation} by taking $\mathcal{T}_1=k^{-1}s_0=\mathcal{O}_+(\rho_{\infty}^{\frac{\gamma-1}{\gamma}}).$ 

We next analyze the variation of $\varrho$ in the interval $(b_0,s_0).$ The monotonic relation in Lemma \ref{lem:Monotonic} implies that 
\begin{equation}\label{eq:MontonicR1}
0\leq\sigma-\vartheta(\sigma)\leq s_0-\vartheta(s_0), ~~\varrho(\sigma)\ge \varrho(s_0)>0, 
\end{equation}
which together with \eqref{eq:RHC}, leads to 
\begin{equation}\label{eq:MontonicR2}
(\vartheta(\sigma)-\sigma)^2-\gamma\varrho^{\gamma-1}(\sigma)\leq(\vartheta(s_0)-s_0)^2-\gamma\varrho^{\gamma-1}(s_0)<0.
\end{equation}
Thus, substituting \eqref{eq:MontonicR1} and \eqref{eq:MontonicR2} into the first formula in \eqref{eq:EE3} yields that for $\rho_\infty$ close enough to zero,
\begin{equation}\label{eq:Ew9}
\begin{aligned}
\big|\log(\varrho)_{\sigma}\big|\leq \frac{2(s_0-\vartheta(s_0))}{\big|(\vartheta(s_0)-s_0)^2-\gamma\varrho^{\gamma-1}(s_0)\big|}.
\end{aligned}
\end{equation}
Moreover, noting \eqref{eq:E7}-\eqref{eq:E6}, we have that the right part in \eqref{eq:Ew9} equals $\mathcal{O}_+(1).$ That is, there exists $\mathcal{T}_2=\mathcal{O}_+(1)$ such that 
\begin{equation}\label{eq:E9}
    \begin{aligned}
        \big|\log(\varrho)_{\sigma}\big|\leq \mathcal{T}_2.
    \end{aligned}
\end{equation}

Together with $\varrho_\sigma<0$ in Lemma \ref{lem:Monotonic} again, integrating \eqref{eq:E9} and substituting \eqref{eq:E6} lead to  
\begin{equation}
   1<\frac{\varrho(\sigma)}{\varrho(s_0)}\leq \exp{\mathcal{T}_3},
\end{equation}
for some $\displaystyle \mathcal{T}_3=\mathcal{O}_+(\rho_{\infty}^{\frac{\gamma-1}{\gamma}}),$ 
which together with \eqref{eq:E8} implies that for $\sigma\in(b_0,s_0),$
\begin{equation}
    \varrho(\sigma)=\mathcal{O}_+(\rho_\infty^{\frac{1}{\gamma}}).
\end{equation}
The proof is complete.
\end{proof}

We further derive the derivative estimates for the self-similar flow field. 
\begin{thm}[First-order estimates for the self-similar flow field]\label{thm:DerivativeEstSelSimilarSolution}
    For any given $s_0\in(\kappa_1,\kappa_2)$ with $\kappa_1,\kappa_2$ given in Theorem \ref{thm1}, let $(\varrho(\sigma),\vartheta(\sigma)), \sigma\in(s_0,b_0)$ be the $C^1$ solution of problem \eqref{eq:Constantmovingspeedcase}. For $\rho_{\infty}$ close to zero, there exist $\displaystyle\mathcal{T}_4=\mathcal{O}_+(\rho_{\infty}^{\frac{1}{\gamma}}), \mathcal{T}_5=\mathcal{O}_+(1)$ such that 
    \begin{equation}\label{eq:AsympDerivartive}
       \begin{aligned}
            \big|\varrho_{\sigma}\big|\leq \mathcal{T}_4, ~\big|\vartheta_{\sigma}\big|\leq \mathcal{T}_5.
       \end{aligned}
    \end{equation}
\end{thm}
\begin{proof}
    Noting \eqref{eq:MonotonicRelationsforconstantcase} and  substituting \eqref{eq:MontonicR1} and \eqref{eq:MontonicR2} into \eqref{eq:EE3} yield
    \begin{equation}\label{eq:E99}
        \begin{aligned}
            &\big|\varrho_{\sigma}\big|\leq \frac{2\varrho(s_0-\vartheta(s_0))}{\big|(\vartheta(s_0)-s_0)^2-\gamma\varrho^{\gamma-1}(s_0)\big|},\\
            &\big|\vartheta_{\sigma}\big|\leq \frac{2\gamma\varrho^{\gamma-1}}{\big|(\vartheta(s_0)-s_0)^2-\gamma\varrho^{\gamma-1}(s_0)\big|}.
        \end{aligned}
    \end{equation}
    According to \eqref{eq:E8} and \eqref{eq:E6}, we have that for $\rho_\infty$ close to zero,
    \begin{equation}\label{eq:Middle1}
        \big|(\vartheta(s_0)-s_0)^2-\gamma\varrho^{\gamma-1}(s_0)\big|=\mathcal{O}_+(\rho_\infty^{\frac{\gamma-1}{\gamma}}), ~\big|s_0-\vartheta(s_0)\big|=\mathcal{O}_+(\rho_\infty^{\frac{\gamma-1}{\gamma}}).
    \end{equation}
    
    Finally, plugging \eqref{eq:Asympototicrelation} and \eqref{eq:Middle1} into \eqref{eq:E99} yields \eqref{eq:AsympDerivartive}. The proof is complete. 
\end{proof}

As a corollary of Theorem \ref{thm:VarianceofSelfsimilarSolution}  and Theorem \ref{thm:DerivativeEstSelSimilarSolution}, we establish the following $C^1$ estimates on the Riemann invariants $\omega_\pm(\sigma)$ defined for the self-similar flow configuration.   

\begin{thm}\label{thm:C01E}
Let $(\varrho(\sigma),\vartheta(\sigma)), \sigma\in(s_0,b_0)$ be the $C^1$ solution of problem \eqref{eq:Constantmovingspeedcase}, with given $s_0\in(\kappa_1,\kappa_2)$. 
    Let $$\displaystyle\omega_{\pm}(\sigma):=\vartheta\pm\frac{2\sqrt{\gamma}}{\gamma-1}\varrho^{\frac{\gamma-1}{2}}$$ be the Riemann invariants defined for the self-similar flow configuration. Then, for $\rho_{\infty}$ close enough to zero, there exist $\displaystyle\mathcal{T}_6=\mathcal{O}_+(\rho_\infty^{\frac{\gamma-1}{\gamma}}),  ~\mathcal{T}_7=\mathcal{O}_+(1)$ such that 
    \begin{equation}\label{eq:RGHH}
    \begin{aligned}
        \big|\omega_{\pm}(s_0)-\omega_{\pm}(\sigma)\big|\leq \mathcal{T}_6, \big|t\partial_r\omega\big|\leq \mathcal{T}_7.
    \end{aligned}
    \end{equation}
\end{thm}
\begin{proof}
    By Theorem \ref{thm:VarianceofSelfsimilarSolution} and Theorem \ref{thm:DerivativeEstSelSimilarSolution}, a direct computation shows that for $\rho_\infty$ close to zero,
    \begin{equation}\label{eq:RT1}
        \begin{aligned}
        \big|\frac{\dif \omega_\pm(\sigma)}{\dif \sigma}\big|&=\big|\vartheta_\sigma\pm\sqrt{\gamma}\varrho^{\frac{\gamma-3}{2}}\varrho_\sigma\big|\leq\big|\vartheta_\sigma\big|+\sqrt{\gamma}\varrho^{\frac{\gamma-3}{2}}\big|\varrho_\sigma\big|\\&\leq \mathcal{T}_5+\sqrt{\gamma}\frac{\mathcal{T}_4}{\varrho^{\frac{3-\gamma}{2}}} =\mathcal{O}_+(1).
        \end{aligned}
    \end{equation}

According to $\sigma=r/t,$ there holds that 
\begin{equation}\label{eq:RH6}
    \big|t\partial_r\omega\big|=\big|t\partial_r\sigma\frac{\mathsf{d \omega_\pm}}{\dif \sigma}\big|=\big|\frac{\mathsf{d \omega_\pm}}{\dif \sigma}\big|,
\end{equation}
which together with \eqref{eq:RT1} implies the second formula in \eqref{eq:RGHH} by taking $\displaystyle\mathcal{T}_7=\mathcal{T}_5+\sqrt{\gamma}\frac{\mathcal{T}_4}{\varrho^{\frac{3-\gamma}{2}}}=\mathcal{O}_+(1).$

 Noting \eqref{eq:RT1}, a direct computation shows that for $\rho_\infty$ close to zero,
\begin{equation}\label{eq:RiemanInvarianceSelf}
    \begin{aligned}
        \big|\omega_{\pm}(s_0)-\omega_{\pm}(\sigma)\big|\leq\int_\sigma^{s_0}\big|\frac{\dif\omega_\pm(\eta)}{\dif\eta}\big|\dif\eta\leq \mathcal{T}_7\big|s_0-b_0\big|,
    \end{aligned}
    \end{equation}
    which together with \eqref{eq:Asympototicrelation} and $b_0=\vartheta(b_0)$, implies the first formula in \eqref{eq:RGHH} by taking $\displaystyle\mathcal{T}_6=\mathcal{T}_1\mathcal{T}_7=\mathcal{O}_+(\rho_{\infty}^{\frac{\gamma-1}{\gamma}}).$
  The proof is complete. 
\end{proof}

\section{Asymptotic expansion of the solution near the shock}
     For given $\displaystyle s'(t),$ we can get $(\rho,v)\big|_{\mathsf{S}}$ from \eqref{EE4}, and we in this section figure out their order about $\rho_{\infty}$ for $\rho_{\infty}$ close to zero.
     
     For $\rho>0,$ system \eqref{EE1} is strictly hyperbolic system with two distinct eigenvalues $\lambda_{\pm}$ given by
     \begin{equation}\label{DefofLam}
     \lambda_{\pm}=v\pm c,
     \end{equation}
     where $c$ is the sound speed given by 
     \begin{equation}\label{eq:DefofSonicspeed}
     c=\sqrt{\frac{\dif \rho^\gamma}{\dif\rho}}=\sqrt{\gamma \rho^{\gamma-1}}.
     \end{equation}
    Let $w_{\pm}$ be Riemann invariants as
     \begin{equation}\label{Defofw}
     w_{\pm}=v\pm\frac{2}{\gamma-1}c.
     \end{equation}
     Direct computation shows that for $\rho>0,$ there holds the following,
     \begin{equation}\label{Rep}
     \left\{
     \begin{aligned}
     &v=\frac{w_{-}+w_{+}}{2},\\
     &\rho=(\frac{(\gamma-1)^2(w_{+}-w_{-})^2}{16\gamma })^{\frac{1}{\gamma-1}},\\
     &\lambda_{+}=\frac{\gamma+1}{4}w_{+}+\frac{3-\gamma}{4}w_{-},\\
     &\lambda_{-}=\frac{3-\gamma}{4}w_{+}+\frac{\gamma+1}{4}w_{-}.
     \end{aligned}
     \right.
     \end{equation}
     That is, for $\rho>0,$ $\rho,v$ and $\lambda_{\pm}$ are also functions of $w_{\pm}$.
       
       We have obtained for constant shock speed case the solvability of Rankine-Hugoniot condition and entropy condition and established the asymptotic expansion of solution on shock in Section 2. We can deduce the solvability result and the asymptotic expansion in the same way for the non-constant shock speed case as follows.   
       
       \begin{lemma}\label{lem1}
       Suppose  (A1)  holds. Then, for $\rho_{\infty}$ close enough to zero,  there exists unique $(\rho_s,v_s)$ solving \eqref{EE4}. Furthermore, there hold that
      \begin{equation}\label{eq:rhovEstiamtesS}
       \rho_s=\mathcal{O}_{+}(\rho^{\frac{1}{\gamma}}_{\infty}),~~ s'(t)-v_s=\mathcal{O}_{+}(\rho^{\frac{\gamma-1}{\gamma}}_{\infty}),
       \end{equation}
       and 
       \begin{equation}\label{EE52}
       \left\{
       \begin{aligned}
       &c_s=\mathcal{O}_{+}(\rho^{\frac{\gamma-1}{2\gamma}}_{\infty}),\\
       &s^{\prime}(t)-\lambda_{-s}=\mathcal{O}_{+}(\rho^{\frac{\gamma-1}{2\gamma}}_{\infty}),\\&\lambda_{+s}-s^{\prime}(t)=\mathcal{O}_{+}(\rho^{\frac{\gamma-1}{2\gamma}}_{\infty}).
       \end{aligned}
       \right.
       \end{equation}
       Moreover, if 
       \begin{equation}\label{Pertrubation1}
       \big|(w_{+},w_{-})(r,t)-(w_{+s},w_{-s})\big|\leq \mathcal{T}_8=\mathcal{O}_{+}(\rho^{\beta}_{\infty}), ~\beta>\frac{\gamma-1}{2\gamma},
       \end{equation}
        then there hold that 
        \begin{equation}\label{eq:pertubaedrhovEstiamtesS}
            \rho(r,t) = \mathcal{O}_+(\rho_\infty^\frac{1}{\gamma}),~s'(t)-v(r,t)=\mathcal{O}_{+}(\rho^{\frac{\gamma-1}{\gamma}}_{\infty}), 
        \end{equation}
        and 
       \begin{equation}\label{E12}
       \left\{
       \begin{aligned}
       &\lambda_{+}(r,t)-s^{\prime}(t)=\mathcal{O}_{+}(\rho^{\frac{\gamma-1}{2\gamma}}_{\infty}),\\
       &s^{\prime}(t)-\lambda_{-}(r,t)=\mathcal{O}_{+}(\rho^{\frac{\gamma-1}{2\gamma}}_{\infty}).
       \end{aligned}
       \right.
       \end{equation} 
        Here and in the sequel, we set 
       \begin{equation}
       T_s=T(s(t),t)
       \end{equation}
       with $T\in\{\rho,v,p,c,\lambda_{\pm},w_{\pm}\}.$
       \end{lemma}

       \begin{proof}
          The solvability of \eqref{EE4} follows directly as a corollary of Lemma \ref{lem:RHC}. The asymptotic expansions in \eqref{eq:rhovEstiamtesS} are immediate corollaries of \eqref{eq:E8} and \eqref{eq:E6} established in the proof of Theorem \ref{thm:VarianceofSelfsimilarSolution}. Furthermore, substituting \eqref{eq:rhovEstiamtesS} into the explicit expressions for $\lambda_\pm$ and $c$ given in \eqref{DefofLam} and \eqref{eq:DefofSonicspeed} directly yields \eqref{EE52}. Finally, we proceed to prove \eqref{eq:pertubaedrhovEstiamtesS} and \eqref{E12}.

           Referring to \eqref{Rep},  a direct computation shows that 
           \begin{equation}
               \begin{aligned}
                   &\lambda_+(r,t)-s'(t)=\lambda_{+}(r,t)-\lambda_{+s}+\lambda_{+s}-s'(t)\\&=\frac{\gamma+1}{4}(w_+-w_{+s})+\frac{3-\gamma}{4}(w_--w_{-s})+\lambda_{+s}-s'(t),
               \end{aligned}
           \end{equation}
     which, together with \eqref{EE52} and \eqref{Pertrubation1}, implies the first formula in \eqref{E12}. The second formula in \eqref{eq:pertubaedrhovEstiamtesS} and the third formula in \eqref{E12} follow in the same way. It remains to prove the first formula in \eqref{eq:pertubaedrhovEstiamtesS}.
     
     According to the second formula in \eqref{Rep}, a direct computation shows that 
     \begin{equation}\label{E13}
     	\begin{aligned}
     	\big|\rho-\rho_s\big|&=(\frac{\gamma-1}{4\sqrt{\gamma}})^{\frac{2}{\gamma-1}}\big|(w_{+}-w_{-})^{\frac{2}{\gamma-1}}-(w_{+s}-w_{-s})^{\frac{2}{\gamma-1}}\big|\\&=(\frac{\gamma-1}{4\sqrt{\gamma}})^{\frac{2}{\gamma-1}}(w_{+s}-w_{-s})^{\frac{2}{\gamma-1}}\big|(1+\frac{w_{+}-w_{-}-w_{+s}+w_{-s}}{w_{+s}-w_{-s}})^{\frac{2}{\gamma-1}}-1\big|.
     	\end{aligned}
     	\end{equation}
     By Lagrange mean value theorem, it holds that
     	\begin{equation}\label{eq:Lagrangemeanvalueresults}
     	(1+\frac{w_{+}-w_{-}-w_{+s}+w_{-s}}{w_{+s}-w_{-s}})^{\frac{2}{\gamma-1}}-1=\frac{2}{\gamma-1}(1+\xi_0)^{\frac{3-\gamma}{\gamma-1}}~\frac{w_{+}-w_{-}-w_{+s}+w_{-s}}{w_{+s}-w_{-s}}
     	\end{equation}
     	 for some $\xi_0$ with $\displaystyle\big|\xi_0\big|\in(0,\big|\frac{w_{+}-w_{-}-w_{+s}+w_{-s}}{w_{+s}-w_{-s}}\big|).$ 
         Thus, plugging \eqref{eq:Lagrangemeanvalueresults} into \eqref{E13} leads to 
     	 \begin{equation}\label{06191}
     	 \big|\rho-\rho_s\big|=\frac{2}{\gamma-1}(\frac{\gamma-1}{4\sqrt{\gamma}})^{\frac{2}{\gamma-1}}(1+\xi_0)^{\frac{3-\gamma}{\gamma-1}}(w_{+s}-w_{-s})^{\frac{2}{\gamma-1}}\big|\frac{w_{+}-w_{-}-w_{+s}+w_{-s}}{w_{+s}-w_{-s}}\big|.
     	 \end{equation}
     	
       Note $\gamma\in(1,3)$ and
      \begin{equation}\label{eq:RecallGamma}
      \begin{aligned}
      &w_{+s}-w_{-s}=\frac{4}{\gamma-1}c_s=\mathcal{O}_{+}(\rho^{\frac{\gamma-1}{2\gamma}}_{\infty}),\\
      &\big|w_{+}-w_{-}-w_{+s}+w_{-s}\big|\leq \mathcal{T}_8.
      \end{aligned}
      \end{equation}
      Substituting \eqref{eq:RecallGamma} into \eqref{06191} yields that for sufficiently small positive $\rho_{\infty},$
      \begin{equation}
      \big|\rho-\rho_s\big|\leq  
      \gamma^{\frac{1}{1-\gamma}}2^{\frac{4-2\gamma}{\gamma-1}}c_s^{\frac{3-\gamma}{\gamma-1}} \mathcal{T}_8 =\mathcal{O}_{+}(\rho_{\infty}^{\beta+\frac{3-\gamma}{2\gamma}}),
      \end{equation}
      which together with \eqref{eq:rhovEstiamtesS} and $\beta>\frac{\gamma-1}{2\gamma}$, implies the forth formula in \eqref{eq:pertubaedrhovEstiamtesS}. 
      The proof is complete.	
\end{proof}

      So far, as shown in Lemma \ref{lem1}, we complete the zeroth-order estimates of the Cauchy data on the given shock $\mathsf{S}$. We next estimate its derivatives on $\mathsf{S}$.
      
     \begin{lemma}\label{lem4}
     	Suppose that (A1) and (A3) holds. Then, for $\rho_\infty$ close to zero, there exists $\displaystyle\mathcal{T}_9=\mathcal{O}_+(\rho_\infty^{\varpi_*})$ with $\displaystyle \varpi_*=\min\{0, \frac{1-\gamma}{2\gamma}+\varpi_0\}$, such that 
        \begin{equation}\label{EE53}
          \big|t\partial_rw_{\pm }\big|_\mathsf{S}\big|\leq \mathcal{T}_9.
      \end{equation}
     \end{lemma} 
     
     \begin{proof}
     We divide the proof into two parts. The first one is to calculate $\displaystyle\frac{dw_{\pm}(s(t),t)}{dt}.$ To this end, referring to \eqref{eq:E1}, we derive from the Rankine-Hugoniot condition \eqref{EE4} that 
    \begin{equation}\label{eq:RF4}
     	\left\{
     	\begin{aligned}
     	&v_s=\sqrt{\frac{(\rho_s-\rho_{\infty})(\rho_s^{\gamma}-\rho^{\gamma}_{\infty})}{\rho_{\infty}\rho_s}},\\
     	&v_s=\frac{(\rho_s-\rho_{\infty})}{\rho_s}s'(t),\\
        &\rho_s>\rho_\infty,
     	\end{aligned}
     	\right.
     	\end{equation}
which gives that
     \begin{equation}\label{EE5}
         (s'(t))^2=\frac{\rho_s(\rho_s^\gamma-\rho_\infty^\gamma)}{(\rho_s-\rho_\infty)\rho_\infty}.
     \end{equation}
     Differentiating \eqref{EE5} with respect to $t$, we get that 
     	\begin{equation}\label{E14}
     	\begin{aligned}
     	2s^{\prime}(t)s''(t)=\frac{\rho_s^\gamma-\rho_\infty^\gamma}{\rho_s-\rho_{\infty}}\frac{1}{\rho_{\infty}}\frac{\dif\rho_s}{\dif t}+\frac{\rho_s}{\rho_{\infty}}\frac{\gamma\rho_s^{\gamma-1}}{\rho_s-\rho_{\infty}}\frac{\dif\rho_s}{\dif t}-\frac{\rho_s}{\rho_{\infty}}\frac{\rho_s^\gamma-\rho_\infty^\gamma}{(\rho_s-\rho_{\infty})^2}\frac{\dif\rho_s}{\dif t}.
     	\end{aligned}
     	\end{equation}
        Note that 
        \begin{equation}\label{Ad}
     	\frac{\rho_s^\gamma-\rho_\infty^\gamma}{\rho_s-\rho_{\infty}}\frac{1}{\rho_{\infty}}-\frac{\rho_s}{\rho_{\infty}}\frac{\rho_s^\gamma-\rho_\infty^\gamma}{(\rho_s-\rho_{\infty})^2}=-\frac{\rho_s^\gamma-\rho_\infty^\gamma}{\rho_s-\rho_{\infty}}\ \frac{1}{\rho_s-\rho_{\infty}}.
     	\end{equation}  
        Given $(A1)$ and the fact $\displaystyle\rho_s=\mathcal{O}_{+}(\rho_{\infty}^{\frac{1}{\gamma}})$ shown in \eqref{eq:rhovEstiamtesS}, substituting  \eqref{Ad} into \eqref{E14} yields 
        \begin{equation}\label{E15}
     	\big|\frac{\dif\rho_s}{\dif t}\big|=\mathcal{T}_{10}\big|s''(t)\big|,
     	\end{equation}  
      for some $\displaystyle\mathcal{T}_{10}=\mathcal{O}_+(\rho_\infty^{\frac{1}{\gamma}}).$  In addition, differentiating the second equation in \eqref{eq:RF4} with respect to $t$ and substituting \eqref{E15} lead to 
     	\begin{equation}\label{E16}
     	\big|\frac{\dif v_s}{\dif t}\big|=\big|s''(t)-\frac{\rho_{\infty}}{\rho_s}s''(t)+s^{\prime}(t)\frac{\rho_{\infty}}{\rho^2_s}\frac{\dif\rho_s}{\dif t}\big|=\mathcal{T}_{11}\big|s''(t)\big|,
     	\end{equation}
        for some $\mathcal{T}_{11}=\mathcal{O}_+(1).$ 
        
        Thus, combining  \eqref{E15} and \eqref{E16} yields that  
     	\begin{equation}\label{E17}
     	\big|\frac{\dif w_{\pm s}}{\dif t}\big|=\big|\frac{\dif v_s}{\dif t}\pm\frac{2}{\gamma-1}\rho_s^{\frac{\gamma-3}{2}}\frac{\dif \rho_s}{\dif t}\big|=\mathcal{T}_{12}\big|s''(t)\big|,
     	\end{equation}
        for some $\mathcal{T}_{12}=\mathcal{O}_+(1).$

        Note 
     	\begin{equation}\label{E18}
     	\begin{aligned}
     	    \frac{\dif w_{\pm s}}{\dif t}=\partial_tw_{\pm}\big|_{\mathsf{S}}+s^{\prime}(t)\partial_rw_{\pm}\big|_{\mathsf{S}},
     	\end{aligned}
     	\end{equation}
        which, combined with \eqref{E10}, yields that 
        \begin{equation}\label{E19}
          \partial_rw_{\pm}\big|_{\mathsf{S}}=\qnt{\frac{\dif  w_{\pm s}}{\dif t}\mp \frac{\gamma-1}{4}\frac{w^2_{+s}-w^2_{-s}}{s(t)}}\frac{1}{s'(t)-\lambda_{\pm s}}.
        \end{equation}
     Then, recalling the definition of $w_\pm$ and substituting  \eqref{eq:rhovEstiamtesS}, \eqref{EE52}, \eqref{E17} into \eqref{E19} yields that 
     \begin{equation}\label{AA1}
         \big|\partial_rw_{\pm}\big|_{\mathsf{S}}\big|\leq \mathcal{T}_{13}\big|s''(t)\big|+\mathcal{T}_{14} t^{-1}
     \end{equation}
      for some $\mathcal{T}_{13}=\mathcal{O}_+(\rho_\infty^{\frac{1-\gamma}{2\gamma}}), \mathcal{T}_{14}=\mathcal{O}_+(1).$ 

      Finally, due to (A3), we derive from \eqref{AA1} that there exists $\mathcal{T}_9=\mathcal{O}_+(\rho_\infty^{\varpi_*})$ with  $\displaystyle \varpi_*=\min\{0, \frac{1-\gamma}{2\gamma}+\varpi_0\}$ such that for $\rho_\infty$ close to zero, 
      \begin{equation}
          \big|t\partial_rw_{\pm}\big|_{\mathsf{S}}\big|\leq \mathcal{T}_9.
      \end{equation}
The proof is complete.
\end{proof}
     
     \section{proof of Theorem \ref{thm1}}
     
     In the section, we analyze the states in the region between the shock and the piston along characteristic lines issuing from the given shock $\mathsf{S}$. To this end, system \eqref{EE1} can be reduced to the following equivalent form for $\rho>0,$
     \begin{equation}\label{E10}
     \left\{
     \begin{aligned}
     &\partial_t(w_{+})+\lambda_{+}\partial_r(w_{+})=\frac{\gamma-1}{4r}(w^2_{+}-w^2_{-}),\\&
     \partial_t(w_{-})+\lambda_{-}\partial_r(w_{-})=-\frac{\gamma-1}{4r}(w^2_{+}-w^2_{-}),
     \end{aligned}
     \right.
     \end{equation}
     where Riemann invariants $w_{\pm}$ are defined in \eqref{Defofw} and the eigenvalues $\lambda_{\pm}$ are defined in \eqref{DefofLam}. Furthermore, setting the notations $$w_{\pm,r}=\partial_rw_{\pm},$$ differentiating \eqref{E10} with respect to $r$ yields 
     \begin{equation}\label{E11}
     \left\{
     \begin{aligned}
     &\partial_t(w_{+,r})+\lambda_{+}\partial_r(w_{+,r})+\frac{\gamma+1}{4}w^2_{+,r}+\frac{3-\gamma}{4}w_{+,r}w_{-,r}\\&=-\frac{\gamma-1}{4r^2}(w^2_{+}-w^2_{-})+\frac{\gamma-1}{2r}(w_{+}w_{+,r}-w_{-}w_{-,r}),\\
     &\partial_t(w_{-,r})+\lambda_{-}\partial_r(w_{-,r})+\frac{\gamma+1}{4}w^2_{-,r}+\frac{3-\gamma}{4}w_{+,r}w_{-,r}\\&=\frac{\gamma-1}{4r^2}(w^2_{+}-w^2_{-})-\frac{\gamma-1}{2r}(w_{+}w_{+,r}-w_{-}w_{-,r}).
     \end{aligned}
     \right.
     \end{equation}

      For fixed positive constant $\mathsf{T},$ we define 
     \begin{equation}\label{338}
     \Omega_{\mathsf{T}}:=\{(r,t): s(t;\mathsf{T})<r<s(t),t>0\},
     \end{equation}
     where
     \begin{equation}\label{339}
     s(t;\mathsf{T}):=b(t)I(t;\mathsf{T})+\{s(t)+b(\mathsf{T})-s(\mathsf{T})\}\{1-I(t;\mathsf{T})\},
     \end{equation}
     \begin{equation}
     I(t;\mathsf{T})=\left\{
     \begin{aligned}
     &1, t\in[0,\mathsf{T}),\\
     &0, t\in [\mathsf{T},+\infty).
     \end{aligned}
     \right.
     \end{equation} 
    
     \begin{figure}
     	\centering
     	\includegraphics[width=0.7\linewidth]{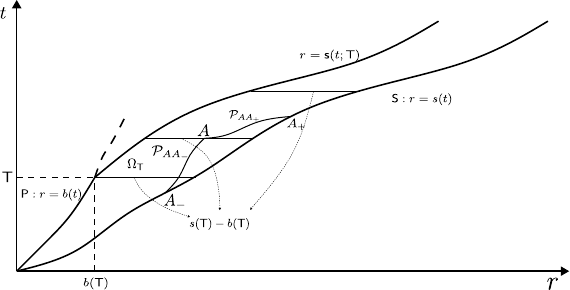}
     	\caption{ $\Omega_{\mathsf{T}}$ and characteristic curves therein }
     	\label{fig:piston-problem-characteristic-curve}
     \end{figure}

     We next study the behavior of characteristics in the stripe region  $\Omega_{\mathsf{T}}$, which significantly contributes to control the variation of $w_{\pm,r}$ when integrating \eqref{E11} to complete the $C^1$ a priori estimates shown in Lemma \ref{lem5}. 
     
     \begin{lemma}\label{lem3}
     	For $\rho_\infty$ close to zero, suppose that (A1) holds and that $(\rho,v)\in C^1(\Omega_{ \mathsf{T}})$ and $ b(t)\in C^2(0,\mathsf{T})$ satisfy \eqref{EE1}, \eqref{EE3} and \eqref{EE4} with 
     	\begin{equation}\label{EE41}
     	\big|w_{\pm}(r,t)-w_{\pm}(s(t),t)\big|\leq \mathcal{R}_1,
     	\end{equation}
     	for some $\displaystyle \mathcal{R}_1=\mathcal{O}_{+}(\rho^{\beta}_{\infty}), \beta>\frac{\gamma-1}{2\gamma},$ then there exists $\displaystyle\mathcal{R}_2=\mathcal{O}_{+}(\rho_{\infty}^{\frac{\gamma-1}{2\gamma}})$ such that when $\rho_\infty$ sufficiently close to zero, for any point $A(r,t)\in\Omega_{\mathsf{T}},$ 
        \begin{equation}\label{EE46}
     	0<t-t_{-}\leq \mathcal{R}_2 t,\quad 0<t_{+}-t\leq \mathcal{R}_2 t.
     	\end{equation}   
        Here and in the sequel, we denote $\mathcal{P}_{AA_{+}}$ as the forward $\lambda_+-$ characteristics  originating from $A$ and terminating at  $A_+(s(t_{+}),t_{+})\in S;$ and denote $\mathcal{P}_{AA_{-}}$ as the backward $\lambda_--$ characteristics  originating from $A$ and terminating at  $A_-(s(t_{-}),t_{-})\in S;$ see Figure \ref{fig:piston-problem-characteristic-curve}.
     \end{lemma}
     \begin{proof}
        Due to  \eqref{EE41}, \eqref{eq:pertubaedrhovEstiamtesS} and \eqref{E12} in Lemma \ref{lem1} hold. Moreover, \eqref{eq:pertubaedrhovEstiamtesS} together with the fact that $\displaystyle b'(t)=v(b(t),t),$ implies for $t\in(0,\mathsf{T}),$ 
     	\begin{equation}\label{EE42}
     	 s'(t)-b'(t)=\mathcal{O}_{+}(\rho^{\frac{\gamma-1}{\gamma}}_{\infty}).
     	\end{equation}
        Thus, integrating \eqref{EE42} yields that there exists $\displaystyle\mathcal{T}_{15}=\mathcal{O}_+(\rho_\infty^{\frac{\gamma-1}{\gamma}})$ such that
        \begin{equation}\label{EE45}
            0<s(t)-b(t)\leq \mathcal{T}_{15}t.
        \end{equation}

        Recall the definition of $\mathcal{P}_{AA_+}$ and $\mathcal{P}_{AA_-}$ that 
        \begin{equation}\label{defofCharacteristics}
            \begin{aligned}
                &\mathcal{P}_{AA_+}=\{(R_+(\tau),\tau): R'_+(\tau)=\lambda_+(R_+(\tau),\tau), R_+(t)=r, \tau\in(t,t_+)\},\\
                &\mathcal{P}_{AA_-}=\{(R_-(\tau),\tau): R'_-(\tau)=\lambda_-(R_-(\tau),\tau), R_-(t)=r, \tau\in(t_-,t)\},
            \end{aligned}
        \end{equation}
        thus, the definition of the points $A_\pm(s(t_\pm),t_\pm)$ implies that 
     	\begin{equation}\label{E23}
     	\left\{
     	\begin{aligned}
     	&s(t_{\pm})=r+\int_{t}^{t_{\pm}}\lambda_{\pm}(R_{\pm}(\tau),\tau)~\dif \tau,\\
     	&s(t_{\pm})=s(t)+\int_{t}^{t_{\pm}}s^{\prime}(\tau)~\dif \tau,
     	\end{aligned}
     	\right.
     	\end{equation}
     	which, combined with $r\in(b(t),s(t))$ following from $A(r,t)\in \Omega_{ \mathsf{T}},$  yields that 
     	\begin{equation}\label{E24}
     	0<s(t)-r=\int_{t}^{t_{\pm}}\lambda_{\pm}(R_{\pm}(\tau),\tau)-s^{\prime}(\tau)~\dif \tau\leq s(t)-b(t).
     	\end{equation} 
        
     Finally, substituting \eqref{E12} and \eqref{EE45} into the \eqref{E24} yields that there exists $\displaystyle \mathcal{R}_2=\mathcal{O}_{+}(\rho_{\infty}^{\frac{\gamma-1}{2\gamma}})$ such that 
     	\begin{equation}
     	0<t-t_{-}\leq \mathcal{R}_2 t,\quad 0<t_{+}-t\leq \mathcal{R}_2 t.
     	\end{equation}
     	The proof is complete. 	 
     \end{proof} 
     
     Based on Lemma \ref{lem3}, we give the a priori estimates on $\|w_{\pm}\|_{C^0(\Omega_{ \mathsf{T}})}$ and $\|\partial_rw_{\pm}\|_{C^0(\Omega_{ \mathsf{T}})}$  as follows.
     
     \begin{lemma}[$C^1$ a priori estimates]\label{lem5}
     	For $\rho_\infty$ close to zero, suppose (A1) (A3) hold and $(\rho,v)\in C^1(\Omega_{\mathsf{T}}), b(t)\in C^2(0,\mathsf{T})$ satisfying \eqref{EE1}, \eqref{EE3} and \eqref{EE4} with for $(r,t)\in\Omega_ {\mathsf{T}},$ 
     	\begin{equation}\label{EE47}
     	\left\{
     	\begin{aligned}
     	&\big|(w_{-},w_{+})(r,t)-(w_{-},w_{+})(s(t),t)\big|\leq \mathcal{R}_3, \\&
     	\big|t\partial_rw_{-}(r,t)\big|+\big|t\partial_rw_{+}(r,t)\big|\leq \mathcal{R}_4,
     	\end{aligned}
     	\right.
     	\end{equation}
     	where $\displaystyle\mathcal{R}_3=\mathcal{O}_{+}(\rho^{\frac{\gamma-1}{2\gamma}+\frac{\varpi_1}{3}}_{\infty}), \displaystyle\mathcal{R}_4=\mathcal{O}_{+}(\rho^{\frac{1-\gamma}{2\gamma}+\frac{\varpi_1}{2}}_{\infty}), \varpi_1:=\min\{\frac{\gamma-1}{4\gamma}, \varpi_0\}$ and $\varpi_0$ given in Theorem \ref{thm1}, then there exist $\mathcal{R}_5=\mathcal{O}_{+}(\rho^{\frac{\gamma-1}{2\gamma}+\frac{\varpi_1}{2}}_{\infty}), \mathcal{R}_6=\mathcal{O}_{+}(\rho^{\frac{1-\gamma}{2\gamma}+\varpi_1}_{\infty})$ such that for $(r,t)\in\Omega_ {\mathsf{T}},$ 
     	\begin{equation}\label{EE48}
     	\left\{
     	\begin{aligned}
     	&\big|(w_{-},w_{+})(r,t)-(w_{-},w_{+})(s(t),t)\big|\leq \mathcal{R}_5,\\&
     	\big|t\partial_rw_{-}(r,t)\big|+\big|t\partial_rw_{+}(r,t)\big|\leq  \mathcal{R}_6.
     	\end{aligned}
     	\right.
     	\end{equation}
     \end{lemma}
     \begin{proof}
         We first assert that \eqref{EE47} satisfies the conditions of Lemma \ref{lem3}. Consequently, not only does Lemma \ref{lem3} hold, but the conclusions drawn in its proof also apply.
         
         Due to \eqref{EE45} and \eqref{EE47},  a direct computation shows that for $(r,t)\in\Omega_ {\mathsf{T}},$ 
     	\begin{equation}
     	\begin{aligned}
     	&\big|w_{\pm}(r,t)-w_{\pm}(s(t),t)\big|\leq\int_{r}^{s(t)}\big|\partial_\eta w_{\pm}(\eta,t)\big|~\dif \eta\\&\leq\big|s(t)-b(t)\big|\sup_{\eta\in(b(t),s(t))}\big|w_{\pm,\eta}(\eta,t)\big|\\&\leq \mathcal{R}_4\mathcal{T}_{15}=\mathcal{O}_{+}(\rho^{\frac{\gamma-1}{2\gamma}+\frac{\varpi_1}{2}}_{\infty}),
     	\end{aligned}
     	\end{equation}
        where $\displaystyle\mathcal{T}_{15}=\mathcal{O}_+(\rho_\infty^{\frac{\gamma-1}{\gamma}})$ given in  \eqref{EE45}. Thus, we take $\mathcal{R}_5=\mathcal{R}_4\mathcal{T}_{15}$ to prove the first inequality in \eqref{EE48}.
        
        It remains to prove the second inequality in \eqref{EE48}. To this end, for $A(r,t)\in\Omega_ {\mathsf{T}},$ integrating the first equation in \eqref{E11} along $\mathcal{P}_{AA_+}$, and substituting the facts that 
        \begin{equation}
            \big|w_\pm\big|\leq \mathcal{T}_{16}, ~\big|\partial_rw_{\pm}(r,t)\big|\leq \mathcal{R}_4t^{-1},
        \end{equation}
        for some $\displaystyle \mathcal{T}_{16}=\mathcal{O}_+(1),$ yields that for some $\displaystyle \mathcal{T}_{17}=\mathcal{O}_+(1),$
        \begin{equation}\label{E25}
           \big|\partial_rw_{+}(A)-\partial_rw_{+}(A_{+})\big|\leq \mathcal{T}_{17}\int_t^{t_+} \frac{1}{R^2_+(\tau)}+\frac{\mathcal{R}_4}{\tau R_+(\tau)}+\frac{\mathcal{R}_4^2}{\tau^2} \dif \tau, 
        \end{equation}
        where $\displaystyle\mathcal{P}_{AA_+}=\{(R_+(\tau),\tau):\tau>0\}$ given in \eqref{defofCharacteristics}.  Moreover, given $\displaystyle R_+(\tau)\in (b(\tau),s(\tau)),$ it follows from \eqref{EE42} and (A1) that 
        \begin{equation}\label{eq:narrowratio}
            R_+(\tau)/\tau=\mathcal{O}_+(1). 
        \end{equation}
        Thus, substituting \eqref{eq:narrowratio} into \eqref{E25} yields that  
        \begin{equation}
            \begin{aligned}
                &\big|\partial_rw_{+}(A)-\partial_rw_{+}(A_{+})\big|\leq \mathcal{T}_{17}\int_t^{t_+} \frac{1}{\tau^2}+\frac{\mathcal{R}_4}{\tau^2}+\frac{\mathcal{R}_4^2}{\tau^2} \dif \tau\\&\leq\mathcal{T}_{17}\frac{1}{t^2}(1+\mathcal{R}_4+\mathcal{R}_4^2)(t_+-t).
            \end{aligned}
        \end{equation}
        Multiplying the former formula by $t$ and substituting \eqref{EE46} lead to 
        \begin{equation}
            \big|t\partial_rw_+(r,t)\big|\leq \frac{t}{t_+}\big|t_+\partial_rw_+(s(t_+),t_+)\big|+\mathcal{T}_{17}(1+\mathcal{R}_4+\mathcal{R}_4^2)\mathcal{R}_2,
        \end{equation}
        which together with Lemma \ref{lem4} implies that 
        \begin{equation}\label{eq:middleofderivativeestimates1}
            \big|t\partial_rw_+(r,t)\big|\leq \frac{t}{t_+}\mathcal{T}_9+\mathcal{T}_{17}(1+\mathcal{R}_4+\mathcal{R}_4^2)\mathcal{R}_2.
        \end{equation}
        Recall that 
        \begin{equation}
            \mathcal{T}_9=\mathcal{O}_+(\rho_\infty^{\frac{1-\gamma}{2\gamma}+\varpi_0}), ~\mathcal{R}_4=\mathcal{O}_{+}(\rho^{\frac{1-\gamma}{2\gamma}+\frac{\varpi_1}{2}}_{\infty}), ~\mathcal{R}_2=\mathcal{O}_{+}(\rho_{\infty}^{\frac{\gamma-1}{2\gamma}})
        \end{equation}
        and 
        \begin{equation}
            \mathcal{T}_{17}=\mathcal{O}_+(1), ~\varpi_1=\min\{\frac{\gamma-1}{4\gamma}, \varpi_0\}.
        \end{equation}
        which,together with \eqref{EE46}, implies that
        \begin{equation}\label{eq:middleofderivativeestimates2}
            \frac{t}{t_+}\mathcal{T}_9+\mathcal{T}_{17}(1+\mathcal{R}_4+\mathcal{R}_4^2)\mathcal{R}_2=\mathcal{O}_{+}(\rho^{\frac{1-\gamma}{2\gamma}+\varpi_1}_{\infty}).
        \end{equation}

        Combining \eqref{eq:middleofderivativeestimates1} and \eqref{eq:middleofderivativeestimates2}, by taking $$\displaystyle \mathcal{R}_6=\frac{t}{t_+}\mathcal{T}_9+\mathcal{T}_{17}(1+\mathcal{R}_4+\mathcal{R}_4^2)\mathcal{R}_2,$$ we obtain the first formula in \eqref{EE48}. The second formula in \eqref{EE48} follows from the similar argument. The proof is complete.
     \end{proof}

     Now, we are ready to prove the main result Theorem \ref{thm1}.
     
     \begin{proof}[Proof of Theorem \ref{thm1}]	

       We divide the proof into two steps: first, we demonstrate the local solvability of problem \eqref{EE1}-\eqref{EE4}; second, we apply the a priori estimates in Lemma \ref{lem5} to extend the local solution to a global one. 
       
       \textbf{Step 1:} (Local solvability). Due to (A2) that $\displaystyle s''(t)=0, t\in(0,\kappa_3)$, referring to \cite[Lemma 3.5]{Wang2005DCDS} and Remark \ref{rem:localexsitence}, 
       we conclude that problem \eqref{EE1}-\eqref{EE4} admits unique straight piston trajectory $\displaystyle b(t)\in C^2(0,t_0)$ and unique self-similar flow field $\displaystyle (\rho,v)\in C^1(\Omega_{t_0}\cap\{t<t_0\}),$ for some $\displaystyle t_0\in(0,\kappa_3).$ Moreover, recalling the definition of $\displaystyle\mathcal{O}_+(\rho_\infty^\alpha)$ in Definition \ref{def:defofO}, Lemma \ref{thm:C01E} implies that there exists a constant $\epsilon_1>0$ such that when $\displaystyle\rho_\infty\in(0,\epsilon_1),$ the obtained self-similar flow field satisfies 
       \begin{equation}
     	\left\{
     	\begin{aligned}
     	&\big|(w_{-},w_{+})(r,t)-(w_{-},w_{+})(s(t),t)\big|\leq \tilde{\mathcal{R}}_3, \\&
     	\big|t\partial_rw_{-}(r,t)\big|+\big|t\partial_rw_{+}(r,t)\big|\leq \tilde{\mathcal{R}}_4, (r,t)\in \Omega_{t_0}\cap\{t<t_0\},
     	\end{aligned}
     	\right.
     	\end{equation}
       for some $\displaystyle\tilde{\mathcal{R}}_3=\mathcal{O}_{+}(\rho^{\frac{\gamma-1}{2\gamma}+\frac{\varpi_1}{3}}_{\infty}), \tilde{\mathcal{R}}_4=\mathcal{O}_{+}(\rho^{\frac{1-\gamma}{2\gamma}+\frac{\varpi_1}{2}}_{\infty})$ and $\varpi_1$ defined in Lemma \ref{lem5}. 
       
       Furthermore, due to Lemma \ref{lem1} and Lemma \ref{lem4}, there exists a constant $\epsilon_2\in(0,\epsilon_1)$ such that when $\rho_\infty\in(0,\epsilon_2),$ $\displaystyle (\rho,v)\big|_{\mathsf{S}\cap\{t>t_0/2\}}$ is well-defined and satisfies \eqref{eq:rhovEstiamtesS}, \eqref{EE52} and the following 
       \begin{equation}\label{eq:proofofmainthmE1}
           \big|t\partial_rw_{-}(s(t),t)\big|+\big|t\partial_rw_{+}(s(t),t)\big|\leq \frac{1}{10}\tilde{\mathcal{R}}_4,~ t>t_0/2.
       \end{equation}
       Here to derive \eqref{eq:proofofmainthmE1} from \eqref{EE53} in Lemma \ref{lem4}, the fact $\displaystyle \frac{1-\gamma}{2\gamma}+\varpi_0>\frac{1-\gamma}{2\gamma}+\frac{\varpi_1}{2}$ following from the definition of $\varpi_1,$ is applied.

       Therefore, referring to \cite[Chapter 1]{Li1985boundary}, results on the local existence and uniqueness of classical solution for hyperbolic systems with initial data, there exists $\delta>0$ such that the problem \eqref{EE1} with Cauchy data prescribed on $\displaystyle\{(r,t):r=s(t), t>{t_0}/{2}\},$ admits unique classical solution 
       $\displaystyle (\rho,v)\in C^1(\Omega_{t_0}\cap S_{\delta})$ satisfying 
       \begin{equation}
     	\left\{
     	\begin{aligned}
     	&\big|(w_{-},w_{+})(r,t)-(w_{-},w_{+})(s(t),t)\big|\leq \tilde{\mathcal{R}}_3, \\&
     	\big|t\partial_rw_{-}(r,t)\big|+\big|t\partial_rw_{+}(r,t)\big|\leq \tilde{\mathcal{R}}_4, (r,t)\in \Omega_{t_0}\cap \mathsf{S}_{\delta},
     	\end{aligned}
     	\right.
     	\end{equation}
       Here $$\mathsf{S}_{\delta}=\{(r,t):\dist((r,t),\mathsf{S})<\delta,t>2/3t_0\},$$ 
       and $\delta>0$ depends on $\tilde{\mathcal{R}}_3, \tilde{\mathcal{R}}_4,$ and $\displaystyle\inf\{\big|\lambda_\pm(s(t),t)-s'(t)\big|: t>0\}.$ 

       Collecting the obtained $$b(t),t\in(0,t_0), (\rho,v)\in C^1(\Omega_{t_0}\cap\{t<t_0\}), (\rho,v)\in C^1(\Omega_{t_0}\cap S_{\delta})$$ together, we conclude that when $\rho_\infty\in(0,\epsilon_2),$ there exists $t_1\in(0,t_0)$ such that problem \eqref{EE1}-\eqref{EE4} admits $\displaystyle b(t)\in C^2(0,t_1)$ and $\displaystyle (\rho,v)\in C^1(\Omega_{t_1}),$ satisfying that for $(r,t)\in \Omega_{t_1},$
       \begin{equation}\label{eq:proofofmainE2}
     	\left\{
     	\begin{aligned}
     	&\big|(w_{-},w_{+})(r,t)-(w_{-},w_{+})(s(t),t)\big|\leq \tilde{\mathcal{R}}_3, \\&
     	\big|t\partial_rw_{-}(r,t)\big|+\big|t\partial_rw_{+}(r,t)\big|\leq \tilde{\mathcal{R}}_4.
     	\end{aligned}
     	\right.
     	\end{equation}

    \textbf{Step 2:} (Global solvability). Due to Lemma \ref{lem5}, \eqref{eq:proofofmainE2} implies that there exists a constant $\epsilon_3\in(0,\epsilon_2)$ such that when $\rho_\infty\in(0,\epsilon_3),$ the obtained local solution satisfies for $(r,t)\in \Omega_{t_1},$
     \begin{equation}\label{eq:proofofmainE3}
     	\left\{
     	\begin{aligned}
     	&\big|(w_{-},w_{+})(r,t)-(w_{-},w_{+})(s(t),t)\big|\leq \frac{1}{10}\tilde{\mathcal{R}}_3, \\&
     	\big|t\partial_rw_{-}(r,t)\big|+\big|t\partial_rw_{+}(r,t)\big|\leq \frac{1}{10}\tilde{\mathcal{R}}_4.
     	\end{aligned}
     	\right.
     	\end{equation}

    To extend the obtained local solution in step 1, we need to solve problem \eqref{EE1} with Cauchy data prescribed on $\displaystyle\{(r,t):r=s(t;t_1), t>t_1\}$ and free boundary condition $b'(t)=v(b(t),t),t>t_1$ prescribed on unknown piston trajectory $\displaystyle r=b(t), t>t_1$. 
    
    Since the given Cauchy data satisfies  \eqref{eq:proofofmainE3}, and the piston speed is bigger than the eigenvalue $\lambda_-,$ i.e., $v(b(t),t)>\lambda_-(b(t),t)=v(b(t),t)-c(b(t),t),$ the local existence results of hyperbolic system with Cauchy data ensures to extend $\displaystyle b(t)\in C^2(0,t_1), (\rho,v)\in C^1(\Omega_{t_1})$ to $\displaystyle b(t)\in C^2(0,t_1+\delta_1), (\rho,v)\in C^1(\Omega_{t_1+\delta_1})$ satisfying \eqref{eq:proofofmainE2} for some constant $\delta_1>0$ depending on $\tilde{\mathcal{R}}_3, \tilde{\mathcal{R}}_4.$ Moreover, Lemma \ref{lem5} implies \eqref{eq:proofofmainE3} again. 
    
    Therefore, the continuity argument, combined with the a priori estimates in Lemma \ref{lem5}, ensures the global classical solution of problem  \eqref{EE1}-\eqref{EE4}.

    Finally, we take $\epsilon=\epsilon_3$ to complete the proof. 
     \end{proof}
 
 \medskip
 
 \section*{Acknowledgments}
 Qianfeng Li was partially supported by Sino-German (CSC-DAAD) Postdoc Scholarship Program, 2023 (No. 57678375). Yongqian Zhang was partially supported by NSFC Project 11421061 and by NSFC Project 12271507.

 \section*{Declarations}

\noindent\textbf{Conflict of interest} 
On behalf of all authors, the corresponding author states that there is no conflict of interest.

\noindent\textbf{Data Availability} The paper does not use any data set.

\bibliographystyle{plain}
\bibliography{CKWX20240603}
	
\end{document}